\documentclass[12pt,reqno]{amsart}
\usepackage{amsmath,amsthm,amsopn,amssymb}
\usepackage{enumerate}
\usepackage[a4paper, margin={2.5cm}]{geometry}
\usepackage{enumerate}
\usepackage{mathrsfs}

\allowdisplaybreaks

\newtheorem{theorem}{Theorem}[section]
\newtheorem{corollary}[theorem]{Corollary}

\newtheorem{proposition}[theorem]{Proposition}

\theoremstyle{definition}

\newtheorem{remark}[theorem]{Remark}

\numberwithin{equation}{section}

\newcommand{\A}{\mathbf A}

\newcommand{\B}{\mathbf B}

\newcommand{\calD}{\mathcal D}

\newcommand{\cB}{\mathcal B}
\newcommand{\cA}{\mathcal A}
\newcommand{\cL}{\mathcal L}
\newcommand{\cP}{\mathcal P}

\newcommand{\e}{\mathrm e}
\newcommand{\E}[1][]{\mathbb E^{#1}}
\newcommand{\F}{\mathcal F}

\newcommand{\G}{\mathcal G}
\renewcommand{\H}{\mathbf H}
\newcommand{\cH}{\mathcal H}
\newcommand{\N}{\mathbb N}
\renewcommand{\P}[1][]{\mathbb P^{#1}}
\newcommand{\R}{\mathbb R}

\newcommand{\X}{\mathbf X}
\newcommand{\cX}{\mathcal X}

\newcommand{\bfa}{\mathbf a}

\newcommand{\V}{\mathbf V}

\newcommand{\Z}{\mathbf Z}
\newcommand{\bfj}{\mathbf j}

\newcommand{\bl}{\mathbf 1}

\newcommand{\eps}{\varepsilon}
\newcommand{\wh}{\widehat}
\newcommand{\wt}{\widetilde}

\newcommand{\set}[1]{\left\{#1\right\}}

\numberwithin{equation}{section}
\numberwithin{theorem}{section}

\begin{document}

\title{The distribution of the spine of a Fleming-Viot type process}

\author{
{\bf Mariusz Bieniek} \ and \  {\bf Krzysztof Burdzy} }
\address{MB: Intitute of Mathematics, University of Maria Curie Sk\l odowska, 20-031
Lublin, Poland}
\address{KB: Department of Mathematics, Box 354350,
University of Washington, Seattle, WA 98195, USA}

\email{mariusz.bieniek@umcs.lublin.pl}
\email{burdzy@math.washington.edu}

\thanks{Research supported in part by NSF Grant DMS-1206276.}

\begin{abstract}
  We show uniqueness of the spine of a Fleming-Viot particle system under minimal assumptions on the driving process. If the driving process is a continuous time Markov process on a finite space, we show that asymptotically, when the number of particles goes to infinity, the branching rate for the spine is twice that of a generic particle in the system, and every side branch has the distribution of the unconditioned generic branching tree.
\end{abstract}


\maketitle

\section{Introduction}

It is well known that, under  suitable assumptions, a branching process can be decomposed into a spine and side branches. 
A detailed review of the relevant literature is presented in \cite[Sect.~2.2]{EK2004}. The ``Evans' immortal particle picture'' was introduced in \cite{Evans}. Another key paper in the area is \cite{LPP}.
Heuristically speaking, the spine has the distribution of the driving process conditioned on non-extinction, the side branches have the distributions of the critical branching process, and the branching rate along the spine is twice the rate along any other trajectory.

 We will prove results for the  Fleming-Viot branching process introduced in \cite{BurdzyMarch00} that have the same intuitive content.
 Our results have to be formulated in a way different from the informal desscription given above for two reasons. The first, rather mundane, reason is that the Fleming-Viot branching process has a different structure from the processes considered in \cite[Sect.~2.2]{EK2004}. A more substantial difference is that for a Fleming-Viot process with a fixed (finite) number of particles, the distribution of the spine does not have an elegant description (as far as we can tell). On the top of that, unlike in the case of superprocesses, the limit of Fleming-Viot processes, when the number of  particles goes to infinity, has not been constructed (and might not exist in any interesting sense). Hence, our results will be asymptotic in nature. We will show that the limit of the spine processes, as the number of particles goes to infinity, has the distribution of the driving process conditioned never to hit the boundary. We will also prove that the rate of branching along the spine converges to twice the rate of a generic particle
and  the distribution of a side branch converges to the distribution of a branching process with the limiting branching rate.

Our main results on the asymptotic spine distribution are limited to Fleming-Viot processes driven by continuous time Markov processes on finite spaces. We conjecture that analogous results hold for all Fleming-Viot processes (perhaps under mild technical assumptions).

The paper is organized as follows. Section \ref{sec:intro} contains basic definitions. It is followed by Section \ref{spine} proving existence of the spine under very weak assumptions, thus significantly strengthening a similar result from \cite{GK}. Section \ref{DHP} shows that a historical process, in the spirit of \cite{DP91}, can be represented as a Fleming-Viot process and satisfies an appropriate limit theorem. Section \ref{asymspine} contains the main theorems on the distribution of the spine, its branching rate, and its side branches. Section \ref{fixedN} shows by example that the  results on the spine distribution must have asymptotic character because they do not necessarily hold for a process with a fixed number of particles.

\section{Basic definitions}\label{sec:intro}

Our main theorems will be concerned with Fleming-Viot processes driven by Markov processes on finite state spaces. Nevertheless we need to consider Fleming-Viot processes with an abstract underlying state space because our proofs will be based on ``dynamical historical processes'' which are Fleming-Viot processes driven by Markov processes with values in function spaces.  

Let $E$ be a topological space and let $F$ be a Borel proper subset of $E$. We will write
$F^c= E\setminus F$. 
Let $Y_t$, $t\geq 0$, be a continuous time strong Markov process with state space $E$
whose almost all sample paths are right continuous.  For $s\geq 0$, let
\begin{equation*}
  \tau_{F,s}=\inf\set{t>s:Y_t\in F^c},
\end{equation*}
and assume that $\tau_{F,s}$ is a stopping time with respect to the natural filtration
of $Y$ for all $s\geq 0$. 
We assume that $F^c$ is absorbing,
i.e., $Y_t = Y_{\tau_{F,s}}$ for all $t\geq \tau_{F,s}$, a.s.

  In most  papers on the Fleming-Viot process, either $Y$ is a diffusion in an open subset $F\subset\R^d$ or $Y$
  is a continuous time Markov process and $E$ is a countable set, so $\tau_F$ is a stopping
  time in those cases.  We recall here that the hitting time of a Borel subset of a topological space by a
  progressively measurable process is a stopping time (see, e.g., Bass \cite{Bass10}).

We will use $\theta$ to denote the usual shift operator but in this section and Section \ref{spine} we do not assume that the transition probabilities of $Y$ are time homogeneous.
We will always make the following assumptions.
\begin{enumerate}[(i)]
\item
$ \P\left( s <\tau_{F,s}<\infty \mid Y_s = x\right)=1 $ for all $x\in F$ and $s\geq 0$.
 \item For every $x\in F$ and $s\geq 0$, the conditional distribution of $\tau_{F,s}$ given $\{Y_s = x\}$ has no atoms. 
\end{enumerate}

Consider an integer $N\geq 2$ and a family $\set{U_k^i,\, 1\leq i\leq N,\, k\geq 1}$ of jointly independent
random variables such that $U_k^i$ has the uniform distribution on the set
$\set{1,\dotsc,N}\setminus\set{i}$. 

We will use induction to construct a Fleming-Viot type process $\X^N_t=(X^1_t,\dotsc,X^N_t)$, $t\geq 0$, with
values in $F^N$. 
Let $\tau_0=0$, suppose that $(X_0^{1,1},\dotsc,X_0^{1,N})\in F^N$, and let 
\begin{equation}\label{o26.1}
  X_t^{1,1},\dotsc,X_t^{1,N},\quad t\geq 0,
\end{equation} 
be independent and have transition probabilities of the process $Y$. We assume that processes in \eqref{o26.1} are independent of the family
$\set{U^i_k,\, 1\leq i\leq N,\, k\geq 1}$. Let 
\begin{equation*}
  \tau_1=\inf\set{t>0:\exists_{1\leq i\leq N}\, X_t^{1,i}\in F^c}.
\end{equation*}
By assumption (ii), no pair of processes can exit $F$ at the same time, so the index $i$ in
the above definition is unique, a.s.

For the induction step, assume that the families 
\begin{equation*}
  X_t^{j,1},\dotsc,X_t^{j,N}, \quad t\geq 0,
\end{equation*}
and the stopping times $\tau_j$ have been defined for $j\leq k$. For each $j\leq k$, denote by $i_j$ the unique
index such that $X_{\tau_j}^{j,i_j}\in F^c$. Let
\begin{equation*}
  X_{\tau_k}^{k+1,m}=X_{\tau_k}^{k,m}\quad \text{for $m\neq i_k$,}
\end{equation*}
and
\begin{equation*}
  X_{\tau_k}^{k+1,i_k}=X_{\tau_k}^{k,U^{i_k}_k}.
\end{equation*}
Let the conditional joint distribution of
\begin{equation*}
  X_t^{k+1,1},\dotsc,X_t^{k+1,N},\quad t\geq \tau_k,
\end{equation*}
given
$\set{X_t^{k+1,m},\, 0\leq t \leq \tau_k, 1\leq m\leq N}$ and $\set{U_k^i,\, 1\leq i\leq N,\, k\geq 1}$,
be that of $N$ independent processes with transition probabilities of $Y$, starting from $X_{\tau_k}^{k+1,m}$, $1\leq m\leq N$.
Let 
\begin{equation*}
  \tau_{k+1}=\inf\set{t>\tau_k:\exists_{1\leq i\leq N}\, X_t^{k+1,i}\in F^c}.
\end{equation*}
We define $\X^N_t:=(X^1_t,\dotsc,X^N_t)$ by
\begin{equation*}
  X_t^m=X_t^{k,m},\qquad \text{  for  } \tau_{k-1}\leq t<\tau_k,\  k\geq 1,\ 
m=1,2,\dotsc,N.
\end{equation*}
Note that the process $\X^N$ is well defined only up to the time
\begin{equation*}
  \tau_\infty :=\lim_{k\to\infty} \tau_k
\end{equation*}
which will be called the lifetime of $\X^N$. We do not assume that
$\tau_{\infty}=\infty$, a.s. 

We will suppress the dependence on $N$ in some of our notation.

\subsection{Dynamical historical processes}\label{ssec:DHPs}

The concept of a dynamical historical process (DHP)  was introduced in \cite[p.~355]{GK}
under a different name. We chose the name ``dynamical historical process'' because the concept of DHP
is based on an intuitive idea similar to the  ``historical process'' (see \cite{DP91}).
Heuristically speaking,
for each $n\in\set{1,\dotsc,N}$, $\{H^n_t(s), 0\leq s\leq t\}$ represents
the unique path in the branching structure of the Fleming-Viot process which goes from $X^n_t$ to one of the points $X^1_0, \dots, X^N_0$ along the trajectories of $X^1, \dots, X^N$ and does not jump at times $\tau_k$. Note that the process $Y$ may have jumps so a dynamical historical process $\{H^n_t(s), 0\leq s\leq t\}$ is not necessarily continuous.

Let $\cA$ be the family of  all sequences of the form $((a_1,b_1), (a_2,b_2), \dots, (a_k,b_k))$, where $a_i \in \{1,\dots, N\}$ and $b_i \in \N$ for all $i$.
For a sequence $\alpha=((a_1,b_1), (a_2,b_2), \dots, (a_k,b_k))$ we will write $\alpha + (m,n)$ to denote
$((a_1,b_1), (a_2,b_2), \dots, (a_k,b_k), (m,n))$. 
We will define a function
$\cL :  \{1,\dots, N\} \times [0,\tau_\infty) \to \cA$.
We interpret $\cL(i,s)$ as a label of $X^i_s$ so, by abuse of notation, we will write $\cL(X^i_s)$ instead of $\cL(i,s)$.
We let $\cL(X^i_s)=((i,0))$ for all $0\leq s < \tau_1$ and $1\leq i\leq N$. If $\cL(X^i_s)=\alpha$ for $\tau_{k-1} \leq s < \tau_k$,
$i \ne i_k$ and $i \ne U^{i_k}_k $ then we let $\cL(X^i_s)=\alpha$ for $\tau_{k} \leq s < \tau_{k+1}$. Suppose that $i = U^{i_k}_k $ and  $\cL(X^i_s)=\alpha$ for $\tau_{k-1} \leq s < \tau_k$. Then we let
$\cL(X^i_s)=\alpha+(i, k)$ and $\cL(X^{i_k}_s)=\alpha+(i_k,k)$ for $\tau_{k} \leq s < \tau_{k+1}$.

Later in the paper we will consider a branching process whose 
individuals are 
elements of the set $\cL(\{1,\dots, N\} \times [0,\tau_\infty))$. A 
sequence 
$\alpha_2$ will be considered an offspring of $\alpha_1$ if  
$\alpha_2 = 
\alpha_1+(m,n)$ for some $m$ and $n$.


Suppose that $\cL(X^n_t)=((a_1,b_1), (a_2,b_2), \dots, (a_k,b_k))$ 
for some 
$k\geq 1$. From the definition of $\cL$ we easily infer that 
$0=b_1<b_2<\dotsc<b_k$ and $\tau_{b_k}\leq t$, so that 
$0<\tau_{b_1}<\dotsc<\tau_{b_k}\leq t<\tau_{b_{k+1}}$. For 
$\tau_{b_m}\leq s<\tau_{b_{m+1}}$ with $1\leq m<k$ we define 
$\chi(n,t,s)=a_m$ and $H^n_t(s)=X^{a_m}_s$, and for $\tau_{b_k}\leq 
s\leq t$ we 
define $\chi(n,t,s)=a_k$ and $H^n_t(s)=X^{a_k}_s$.
Note that $H^n_t(s) = X^{\chi(n,t,s)}_s$ and $\chi(n,t,t) = n$ for 
all $1\leq 
n\leq N$ and $0\leq s\leq t$.

We will call $\{H^n_t(s), 0\leq s\leq t\}$ a dynamical historical 
process (DHP) corresponding to $X^n_t$. Note that $H^n_t$ is defined 
only for $1\leq n \leq N$ and $0\leq t < \tau_\infty$.

We will say that a branching event occurred along $H^k_t$ on the interval $[s_1,s_2]$, where $0 \leq s_1 \leq s_2 \leq t$, if  there exist $s\in [s_1,s_2]$ and  $j\ne k$ such that $\chi(j,t,s) = \chi(k,t,s)$ and $\chi(j,t,s_2) \ne \chi(k,t,s_2)$.

\section{Existence and uniqueness of the spine}\label{spine}

The spine process will be defined below the statement of Theorem \ref{thm:uniquespine}. Roughly speaking, the spine is the unique DHP that extends from time 0 to time $\tau_\infty$.
The existence and uniqueness of the spine
was proved in \cite[Thm.~4]{GK}  under very
restrictive assumptions on the driving process $Y$ and under the assumption that
the lifetime $\tau_\infty$ is infinite. We will prove that
the claim holds under  minimal reasonable assumptions, that
is, the strong Markov property of the driving process and non-atomic
character of the exit time distributions. 

\begin{theorem}\label{thm:uniquespine}
Fix some $N\geq 2$, suppose that $Y$ satisfies assumptions (i)-(ii) in Section 
\ref{sec:intro} and $\X^N_0 \in F^N$, a.s. Then, a.s.,
there exists a unique infinite sequence $((a_1,b_1), (a_2,b_2), \dots)$ such that its every finite initial subsequence is equal to $\cL(X^i_s)$ for 
some $1\leq i \leq N$ and $s\geq 0$. 

\end{theorem}

In the notation of the theorem, we define the spine of $\X^N$ by 
$J(s)=J^N(s) = X^{a_m}_s$ for $\tau_{b_m}\leq s<\tau_{b_{m+1}}$, 
$m\geq 1$. We also let $\cL(J,s) = \cL(X^{a_m}_s)$ and $\chi(J,s) = 
a_m$ for $\tau_{b_m}\leq s<\tau_{b_{m+1}}$, $m\geq 1$.

\begin{proof}[Proof of Theorem \ref{thm:uniquespine}]
{\it Step 1}.
We start with a  simple estimate.
Consider a probability space $(\Omega, \F, \P)$, a $\sigma$-field $\G \subset \F$ and an event $A\in \F$. Suppose that  $\P(A) \geq p$ and let $V = \P( A \mid \G)$. Then
\begin{align*}
      p&\leq \E V= \E\left(V \bl_{\left\{ V<p/2 \right\}}\right)
     +\E\left(V\bl_{ \left\{ V\geq p/2
      \right\} } \right)
      \leq \frac{p}{2}\P\left( V<\frac{p}{2} \right)
      +1\cdot \P\left( V\geq \frac{p}{2} \right)\\
      &\leq \frac{p}{2}
      + \P\left( V\geq \frac{p}{2} \right).
\end{align*}
This implies that
\begin{align}\label{o21.1}
\P(\P( A \mid \G) \geq p/2) = \P(V \geq p/2) \geq p/2.   
\end{align}

In the rest of the proof, $\P$ will refer to the probability
measure on the probability space where $\X^N$ is defined.
Let 
$\tau_F^i=\inf\set{t>0:X_t^i\in F^c}$
and $ m_i=\mathrm{median}\left( \tau_F^i \right)$ for $1\leq i\leq N$.
Note that since the distribution of each $\tau_F^i$ has no atoms, each $m_i$ is uniquely determined. The median $m_i$ depends only on $X^i_0$. 
Let $\bfj$ be the index of the particle with the maximal median of the
exit time from $F$ (we choose the smallest of such indices if there 
is a tie). In other words,
$\bfj$ is the smallest number satisfying
$ m_{\bfj}=\max_{1\leq i\leq N}m_i$.

Recall the definition of $\tau_k$ from Section \ref{sec:intro}.
Let $i^*$ be a function of $i$ defined by $\tau_F^i = \tau_{i^*}$  
and let
\begin{align*}
A' &=   \bigcap_{i\neq \bfj }\left\{\tau_F^i\leq m_{\bfj},
U_{i^*}^{i} = \bfj \right\}, \\
A'' &= \{\tau^{\bfj}_F > m_{\bfj}\},\\
A & = A' \cap A''.
\end{align*}
The following estimate holds for any $\X^N_0 \in F^N$,
\begin{align*}
\P(A') & = \prod_{i\neq \bfj}\left[\P(\tau_F^i\leq m_{\bfj})
      \P\left(U_{i^*}^{i} = \bfj \right)\right]
	   \geq \frac{1}{2^{N-1}}\left( \frac{1}{N-1}\right)^{N-1} =:p>0.
\end{align*}
The events $A' $ and $A''$ are independent and $\P(A'') = 1/2$ so
$\P(A) \geq p/2 =: p_1$.

Let $\F_t = \sigma\{\X^N_s, s\leq t\}$.
Note that all $\tau^i_F$, $1\leq i\leq N$, are distinct, a.s., because the hitting time distributions have no atoms.
Let $\wh \tau^1 < \wh \tau^2 < \dots < \wh\tau^{N-1}$ be the ordering of the set $\{\tau^i_F, i\ne \bfj\}$. 
Let $k(i)$ be defined by $\wh \tau^i = \tau^{k(i)}_F$.

Since $\P (A ) \geq p_1$, we obtain from \eqref{o21.1},
\begin{align*}
\P( \P (A \mid \F_{\wh\tau^1}) \geq p_1/2) \geq p_1/2.
\end{align*}
Let $B_1 = \{\P (A \mid \F_{\wh\tau^1}) \geq p_1/2\}$
and $C_1 = \{\tau^{k(1)}_F \circ \theta_{\wh \tau^1} > m_{\bfj}-\wh \tau^1\}$.
If $B_1$ holds then 
\begin{align*}
\bl_{\left\{U_{k(1)^*}^{k(1)} = \bfj\right\}}=
\P(U_{k(1)^*}^{k(1)} = \bfj \mid \F_{\wh\tau^1}) \geq
\P (A' \mid \F_{\wh\tau^1}) \geq 
\P (A \mid \F_{\wh\tau^1}) \geq p_1/2 >0.
\end{align*}
So if $B_1$ holds then
$\left\{U_{k(1)^*}^{k(1)} = \bfj\right\}$ holds as well. This and the fact that
the processes 
$\{X^{k(1)}_t, t\geq \wh \tau^1\}$ 
and $\{X^{\bfj}_t, t\geq \wh \tau^1\}$ 
are conditionally i.i.d. given $ \F_{\wh\tau^1}$ imply that on the event $B_1$, 
\begin{align*}
\P (C_1 \cap A \mid \F_{\wh\tau^1}) \geq (p_1/2)^2.
\end{align*}
We have $\P(B_1) \geq p_1/2$, so
\begin{align}\label{o23.1}
\P (C_1 \cap A ) \geq (p_1/2)^3.
\end{align}

Next we will apply induction. 
Let
\begin{align*}
p_n & = (p_{n-1}/2)^3, \qquad  n\geq 2, \\
C_n & = \bigcap_{i=1}^n
\{\tau^{k(i)}_F \circ \theta_{\wh \tau^i} > m_{\bfj}-\wh \tau^i\},
\qquad 2 \leq n \leq N-1,\\
B_n & = \{ \P (C_{n-1} \cap A \mid \F_{\wh\tau^{n}}) \geq p_{n}/2\},
\qquad 2 \leq n \leq N-1.
\end{align*}

Suppose that for some $1\leq n\leq N-2$,
\begin{align*}
\P \left(C_n \cap A\right) \geq p_{n+1}.
\end{align*}
Note that the above inequality holds for $n=1$, by \eqref{o23.1}.
This induction assumption and \eqref{o21.1} imply that,
\begin{align*}
\P(B_{n+1} )=
\P\left(\P \left(C_n \cap A \mid \F_{\wh\tau^{n+1}} \right) \geq p_{n+1}/2 \right) \geq p_{n+1}/2.
\end{align*}
If $B_{n+1}$ holds then 
\begin{align*}
\bl_{\left\{U_{k(n+1)^*}^{k(n+1)} = \bfj\right\}}=
\P(U_{k(n+1)^*}^{k(n+1)} = \bfj \mid \F_{\wh\tau^{n+1}}) \geq
\P (A' \mid \F_{\wh\tau^{n+1}}) \geq 
\P (A \mid \F_{\wh\tau^{n+1}}) \geq p_{n+1}/2 >0.
\end{align*}
Hence if $B_{n+1}$ holds then
$\left\{U_{k(n+1)^*}^{k(n+1)} = \bfj\right\}$ holds as well. This and the fact that
the processes 
$\{X^{k(n+1)}_t, t\geq \wh \tau^{n+1}\}$ 
and $\{X^{\bfj}_t, t\geq \wh \tau^{n+1}\}$ 
are conditionally i.i.d. given $ \F_{\wh\tau^{n+1}}$ imply that on the event $B_{n+1}$, 
\begin{align*}
\P (C_{n+1} \cap A \mid \F_{\wh\tau^{n+1}}) \geq (p_{n+1}/2)^2.
\end{align*}
We have $\P(B_{n+1}) \geq p_{n+1}/2$, so
\begin{align*}
\P (C_{n+1} \cap A ) \geq (p_{n+1}/2)^3.
\end{align*}
This finishes  the induction step. We conclude that
\begin{align}\label{o23.2}
\P (C_{N-1} \cap A ) \geq (p_{N-1}/2)^3 =: q,
\end{align}
where $q>0$ depends only on $N$.

  If $C_{N-1} \cap A$ holds then  all DHP paths
  $H^n_{m_\bfj}$, $1\leq n\leq N$, must pass through
  $X^{\bfj}_{\wh\tau^1}$, so they all agree with $H^{\bfj}_{\wh\tau^1}$ on the time interval
  $[0,\wh\tau^1]$ in the sense that $\chi(n, m_\bfj, s) = \chi(\bfj, 
  m_\bfj, s) $
for all $s\in [0,\wh\tau^1]$ and $1\leq n \leq N$.

\bigskip
{\it Step 2}.
  Let $\sigma_0=0$, $\bfj_1 = \bfj$, $m_{\bfj_1,1} = m_\bfj$ and
 $  \sigma_1=m_{\bfj_1,1}\wedge \tau_F^{\bfj_1}$.

Suppose that $\bfj_n$ and $\sigma_n$ have been defined for $n=1, \dots, k$.
We define 
$\tau_F^{i,k+1}=\inf\set{t>\sigma_k:X_t^i\in F^c}$
and $ m_{i,k+1}$ to be the  median of the conditional distribution of
$ \tau_F^{i,k+1} - \sigma_k $ given $\sigma_k$ for $1\leq i\leq N$.
Let 
$\bfj_{k+1}$ be the smallest number satisfying
$ m_{\bfj_{k+1}, k+1}=\max_{1\leq i\leq N}m_{i,k+1}$.
  Let 
  \begin{equation*}
    \sigma_{k+1}=\left(m_{\bfj_{k+1},k+1} + \sigma_k\right)\wedge \tau_F^{\bfj_{k+1}, k+1}.
  \end{equation*}
It is easy to see that $\sigma_k<\tau_\infty$, a.s., for
  all $k\geq 1$.  
Let $C_{N-1}^k$ and $A^k$ be defined in  the way analogous to  $C_{N-1}$ and $A$ but relative to $\bfj_k$ and $m_{\bfj_{k},k}$.
Let $D_k$ be defined by the condition $\bl_{D_k}\equiv \bl_{C_{N-1}^k \cap A^k} \circ\theta_{\sigma_{k-1}}$ for $k\geq 1$.

  By \eqref{o23.2} and the strong Markov property of $\X^N$ applied at time $\sigma_{k-1}$, for each $k\geq 2$,
  \begin{equation*}
    \P(D_k^c\mid D_1^c,\dotsc,D_{k-1}^c)\leq 1-q.
  \end{equation*}
It follows that for every $k\geq 2$,
  \begin{equation*}
    \begin{split}
      \P\left( \bigcup_{i=1}^k D_i\right)&=1-\P(D_1^c)\P(D_2^c\mid D_1^c)\dotsc
      \P(D_k^c\mid D_{1}^c,\dotsc,D_{k-1}^c)\\
      &\geq 1-(1-q)^k.
    \end{split}
  \end{equation*}
  So if $D_\ast=\bigcup_{k=1}^\infty D_k$, then $\P\left( D_\ast \right)=1$, i.e.,
  almost surely at least one of the events $D_k$ occurs.
For any $m \geq 1$, the same claim applies to the process $\X^N$ after time $\tau_m$,
by the strong Markov property,
 so if $G_m=\set{\bl_{D_\ast}\circ\theta_{\tau_m}=1}$,
  then $\P(G_m)=1$ for all $m\geq 1$, and, therefore,
  \begin{equation}\label{eq:all}
    \P\left( \bigcap_{m=1}^\infty G_m \right)=1.
  \end{equation}
Fix any $m\geq 1$. By \eqref{eq:all}, with probability 1, there exists $k$ such that $\bl_{D_k}\circ\theta_{\tau_m}=1$ and we let $k$ denote the smallest integer with this property.
Let
$\eta^{(m)}_k=\sigma_{k}\circ\theta_{\tau_{m}}+\tau_m$, $ k\geq 1$.
The last remark in Step 1 implies that all
 DHP paths
  $H^n_{\eta^{(m)}_k}$, $1\leq n\leq N$, must agree on the time interval
  $[0,\tau_m]$, a.s., that is,
$\chi(j, \eta^{(m)}_k, s) = \chi(1, \eta^{(m)}_k, s) $
for all $s\in [0,\tau_m]$ and $1\leq j \leq N$.
For any
  $t<\tau_\infty$ we find a random $m$ such that $t\leq \tau_m$
and $k$ such that $\bl_{D_k}\circ\theta_{\tau_m}=1$.
Suppose that $\cL\left(X^1_{\eta^{(m)}_k}\right)=((c_1,d_1),  \dots, (c_n,d_n))$.
It is 
easy to check that if we let $a_m = c_m$ and $b_m = d_m$ for $1\leq m \leq n$ such that $\tau_{d_{m+1}}< t$ then this definition is consistent when we vary $t$ over $[0, \tau_\infty)$. We have
defined a unique sequence
$((a_1,b_1), (a_2,b_2), \dots)$ satisfying the theorem.
\end{proof}

\section{Dynamical historical process as  a Fleming-Viot process}\label{DHP}

We will write $\{Y^t, 0\leq s \leq t\}$ to denote the process $Y$ conditioned by $\tau_{F,0} > t$.
In this section, we prove that, as the number of particles $N$ goes to infinity, the empirical distribution of DHPs at time $t$ converges to the distribution of the trajectory of the process
$Y$ conditioned by $\tau_{F,0} > t$. For technical reasons
we impose two extra assumptions on $\X$; they will stay in force for the rest of the paper. We assume that $\tau_\infty= \infty$, a.s., for all $N$, and that the process $Y$ is time homogeneous. 
If the driving process is Brownian motion in a Lipschitz domain with 
the Lipschitz constant less than 1 (\cite{BBF, GK}) or Brownian 
motion in a polytope with $N=2$ (\cite{BBF}), then $\tau_\infty= 
\infty$, a.s. However, it was proved in \cite{BBP} that $\tau_\infty 
< \infty$, a.s., for every $N$, for some Fleming-Viot processes 
driven by one-dimensional diffusions. 
Crucially for the rest of our paper, it is easy to see that if the driving process is a continuous time Markov process on a finite space then $\tau_\infty= \infty$, a.s.

Let $D([0,t],E)$ denote the usual Skorokhod space of cadlag functions with values in $E$. Let
$\H^N_t=(H_t^1,\dotsc,H_t^N)$, where $H_t^k$ is DHP of $X_t^k$. Let $\cX^N_t$ and
$\cH^N_t$ denote empirical distribution of $\X^N_t$ and $\H^N_t$, resp., i.e., 
\begin{align*}
 \cX^N_t(A)&=\frac{1}{N}\sum_{k=1}^N \delta_{X^k_t}(A),\qquad
\text{  for  } A\subset E,\\
  \cH^N_t(A)&=\frac{1}{N}\sum_{k=1}^N \delta_{H^k_t}(A),
\qquad \text{  for  } A \subset D([0,t],E).
\end{align*}
Let $\P[\cX]$ and $\E[\cX]$ denote the probability distribution and the corresponding
expectation for the process $\X^N$, assuming that the empirical distribution
of $\X^N_0$ is $\cX$.

\begin{theorem}\label{thm:convergence}
  Assume that $\cX^N_0\Rightarrow \cX$ as $N\to\infty$ for some probability measure
 $\cX$ on $F$ and $\tau_\infty= \infty$, a.s., for all $N$.
Then for every fixed $t\geq 0$ and continuous bounded function $f$ on $D([0,t],E)$, when $N\to \infty$, in probability,
  \begin{equation*}
    \cH_t^N(f) \to
\E[\cX](f(\{Y^t_s, 0\leq s \leq t\}) ).
  \end{equation*}
\end{theorem}

\begin{proof}
The theorem follows rather easily from a result  of Villemonais \cite{Villemonais13} but we have to reformulate the problem to be able to apply that theorem in our setting. Specifically, we have to represent DHP as a time-homogeneous Markov process.

 For $y\in D([0,t],E)$ and $t\geq 0$ let
\begin{equation*}
  y^{(t)}(s)=
  \begin{cases}
    y(s), & \text{for $s<t$,}\\
    y(t), & \text{for $s\geq t$,}
  \end{cases}
\end{equation*}
and note that $y^{(t)} \in D([0,\infty),E)$.
Let
$Z_t=\left( t,Y^{(t)} \right)$.
Note that the process $Z$ is a time-homogeneous Markov process with trajectories in the space
$D_* := D([0,\infty), \R_+ \times D([0,\infty), E))$.
Let 
\begin{align*}
F_Z = \{  (t, (v_t(s))_{s\geq 0})_{t\geq 0} 
\in D_*: v_t(s) \in F \text {  for all  } s,t \geq 0\}.
\end{align*}
The set $F_Z^c$ is absorbing for $Z$ in the sense that if $Z_s\in F_Z^c$ and $s< t$ then $Z_t \in F_Z^c$; but it is not true that $Z_t = Z_s$.

Let $\hat\H^N_t=(\hat H_t^1,\dotsc,\hat
H_t^N)$, where $\hat H_t^k(s)=(t,H_t^k(s))$ for $s\leq t$
and $\hat H_t^k(s)=(t,H_t^k(t))$ for $s> t$. Then
$\hat \H^N_t$ is a Fleming--Viot process in $\R_+ \times D([0,\infty), E)$ based on $Z$. A ``particle'' in this process jumps to the location of another particle when it hits $F_Z^c$. For $t\geq 0$ we define the empirical distribution $\hat\cH^N_t$ of $\hat\H^N_t$ as
\begin{align*}
  \hat\cH^N_t(A)&=\frac{1}{N}\sum_{k=1}^N \delta_{\hat H^k_t}(
  A),\qquad \text{  for  } A\subset \R_+ \times D([0,\infty), E).
\end{align*}
Note that for $A\subset  D([0,\infty), E)$ we have $\hat\cH^N_t(\R_+\times A) =
\cH^N_t(A)$.

By \cite[Thm.~1]{Villemonais13}, for every fixed $t\geq 0$ and continuous bounded function $g$ on $D_*$, when $N\to \infty$, in probability,
  \begin{equation*}
    \hat\cH_t^N(g) \to
\E(g(Z_t) \mid Z_s\in F_Z,  0\leq s \leq t).
  \end{equation*}
This is essentially the assertion of the theorem, cloaked in a different
formal statement.
\end{proof}

\begin{remark}\label{n13.1}
For later reference we state \cite[Thm.~1]{Villemonais13}. This claim may be also considered a corollary
of Theorem \ref{thm:convergence}.
  Assume that $\cX^N_0\Rightarrow \cX$ as $N\to\infty$ for some probability measure
 $\cX$ on $F$.  Then for every fixed $t\geq 0$ and continuous bounded function $f$ on $E$, when $N\to \infty$, in probability,
  \begin{equation*}
    \cX_t^N(f) \to
\E[\cX](f(Y^t_t) ).
  \end{equation*}
\end{remark}

\section{The asymptotic distribution of the spine}\label{asymspine}

For the remaining part of the paper we  assume that $Y$ is a time-homogeneous continuous-time Markov chain with
finite state space $E=\set{0,1,\dotsc,n}$.
We choose  $\set{1,\dotsc,n}$ to play the role of $F$. We assume that $F$ is a communicating class in the sense that for all $x,y\in F$, there is a positive probability that $Y$ will visit $y$ before hitting 0 if it starts from $x$.
Recall that $J^N_t$ denotes the spine process
defined after the statement of Theorem \ref{thm:uniquespine}.

Recall that $\{Y^t, 0\leq s \leq t\}$ is the process $Y$ conditioned by $\tau_{F,0} > t$.
Let $Y^\infty$ denote the process $Y$ conditioned
never to leave $F$. The process $Y^\infty$ can be described
as the spatial component of the space-time Doob's $h$-process 
obtained from $\{(t,Y_t), t\geq 0\}$  by conditioning
by the parabolic function $h$ which is 0 on $F^c$ and grows
to infinity on $F$. Alternatively, we may define the distribution
of $Y^\infty$
as the limit, as $t\to \infty$, of distributions of $Y^t$. We will
not provide a more formal construction of $Y^\infty$
because it does not
pose any technical challenges in our context.

\begin{theorem}\label{thm:main1}
Consider a probability measure $\cX$ on $F$ and suppose that $\cX^N_0\Rightarrow \cX$ as $N\to\infty$. The distribution of $J^N$ converges to the distribution of
  $Y^\infty$ with the initial distribution $\cX$ when $N\to \infty$.
\end{theorem}

\begin{proof}

Consider any $t> 0$.
By Theorem \ref{thm:convergence}, the empirical distribution of the
      dynamical historical paths of $\X^N_t$ at time $t$ converges to the
      distribution of $\{Y^t(s), 0\leq s \leq t\}$, when $N\to \infty$. 
Since the set $F$ is finite, this implies that for every fixed
      $x\in F$, the empirical distribution of DHPs which end at $x$ at time
      $t$ converges, as $N\to \infty$, to the distribution of 
$\{Y^t(s), 0\leq s \leq t\}$ conditioned by  $\{Y^t(t) =x\}$.

 Fix any $u>0$ and a sequence $s_k \to \infty$.
The results of \cite{DarSen67} (see especially (1.1) and Section 4) can be used to show that for any  $x_k \in F$,
the distributions of $\{Y^{s_k}(s), 0\leq s \leq u\}$ conditioned by 
$\{Y^{s_k}(s_k) = x_k\}$ converge, as $k\to \infty$, to the 
distribution of $\{Y^\infty(s), 0\leq s \leq u\}$.

Since the state space $F$ is finite, we can use the diagonal method to show that for any sequence $N_m$ going to infinity 
we can find a subsequence $N^*_m$  of $N_m$ such that for some $p_{x,k} \geq 0$, we have for all $x$ and $k$,
\begin{align}\label{n28.5}
\lim_{m\to \infty}\P\left(J^{N^*_m}_{s_k} = x\right) = p_{x,k}.
\end{align}

Given $\X^{N^*_m}_{s_k}$,
every DHP which ends at $x$ at time $s_k$ has the same
      probability of being the initial part of the spine.
This observation, \eqref{n28.5} and the earlier remarks on the 
convergence of DHPs and convergence of $\{Y^{s_k}(s), 0\leq s \leq 
u\}$ conditioned by $\{Y^{s_k}(s_k) = x_k\}$ imply that the 
distribution of
$\{J^{N^*_m}_s,  0\leq s \leq u\}$ converges, as $m\to \infty$, to the distribution of $\{Y^\infty(s), 0\leq s \leq u\}$.
Since $u$ is arbitrary and $N^*_m$ is a subsequence of any sequence $N_m$, the theorem follows.
\end{proof}

Let $q_{xy}$ denote elements of the transition rate matrix $Q$ for the process $Y$ and let
\begin{align}\label{n6.1}
\lambda_t =  \sum_{y\in F} \P[\cX](Y^t_t=y) q_{y 0}.
\end{align}
Let $M^m_t$ be the number of times that the process $X^m$ branched
before time $t$. More formally, in the notation of Section \ref{sec:intro}, $M^m_t$ is the number of $k$ such that $\tau_k\leq t$ and $U^{i_k}_k = m$.
Let $M^J_t $ be the number of times that the process $J$ branched
before time $t$. More precisely, let $M^J_t $ be the number of $k$ such that 
$\tau_k\leq t$ and  either $i_k = \chi(J,\tau_k)$ or $U^{i_k}_k = \chi(J,\tau_k)$.

\begin{proposition}\label{n2.1}
Assume that $\cX^N_0\Rightarrow \cX$ as $N\to\infty$ for some probability measure
 $\cX$ on $F$. For every fixed $m$,
  the distribution of $M^m$ converges to the distribution 
  of the Poisson process with variable intensity $\lambda_t$ as $N\to \infty$.
\end{proposition}

\begin{proof}

Every process $M^m$ is a Poisson process with variable random intensity equal to
$\sum_{y\in F} \cX^N_t(y)  q_{y 0}$ at time $t$. Fix any $t>0$.
Definition \eqref{n6.1} together with finiteness of $F$ imply that it 
will suffice to prove that,  for all $\eps_1, p_1>0$,
\begin{align}\label{n13.2}
\limsup_{N\to \infty}
\P\left(\sup_{s\in[0,t]}\sum_{y\in F} \left|\cX^N_s(y) - \P[\cX](Y^s_s=y) \right| \geq \eps_1
\right) \leq p_1.
\end{align}
Since $F$ is finite, it will be enough to prove that,
for all $\eps_1, p_1>0$ and $y\in F$,
\begin{align*}
\limsup_{N\to \infty}
\P\left(\sup_{s\in[0,t]} \left|\cX^N_s(y) - \P[\cX](Y^s_s=y) \right| \geq \eps_1
\right) \leq p_1.
\end{align*}
Suppose to the contrary that there exist $p_1, \eps_1>0$ and $y\in F$  such that
\begin{align*}
\limsup_{N\to \infty}\P\left(\sup_{s\in[0,t]} \left|\cX^N_s(y) - \P[\cX](Y^s_s=y) \right| \geq \eps_1
\right) \geq p_1.
\end{align*}
The set $F$ has cardinality $n$ and $\sum_{y\in F} \cX^N_s(y)=  \sum_{y\in F}\P[\cX](Y^s_s=y) = 1$ so the above assumption implies that 
there exist $p_1, \eps_1>0$ and $y^\ast\in F$  such that
\begin{align}\label{n30.2}
\limsup_{N\to \infty} \P\left(\sup_{s\in[0,t]} \left(\cX^N_s(y^\ast) 
- \P[\cX](Y^s_s=y^\ast) \right) \geq \eps_1/n
\right) \geq p_1.
\end{align}

It is easy to see that for any $\eps_1, p_1 >0$ one can find $\delta \in(0,t)$ so small that for every $y\in F$ and $s\geq 0$ the following holds.
\begin{enumerate}
\item
If the number of  $k$ such that $X^k_s = y$ is greater
than or equal to $j$ then with probability greater than $1-p_1$,
the number of $k$ such that $X^k_u = y$ for all $u\in[s,s+\delta]$
is greater than $j(1-\eps_1/(3n))$. 
\item
For all $u\in[s,s+\delta]$,
\begin{align*}
\left|\P[\cX](Y^s_s=y)
- \P[\cX](Y^u_u=y) \right| \leq \eps_1/(3n).
\end{align*}
\end{enumerate}
Let 
\begin{align*}
T &= \inf\left\{s \geq 0: \cX^N_s(y^\ast) - \P[\cX](Y^s_s=y^\ast) 
\geq \eps_1/n\right\},\\
k_1 &= \inf\{ k \geq 0: k\delta \geq T\}.
\end{align*}
By \eqref{n30.2} and the strong Markov property applied at $T$,
\begin{align*}
\limsup_{N\to \infty} \P\left(\cX^N_{k_1\delta}(y^\ast) - 
\P[\cX](Y^{k_1\delta}_{k_1\delta}=y^\ast)  \geq 2\eps_1/(3n)
\right) \geq p_1(1-p_1).
\end{align*}
Note that $k_1 \delta < 2 t$ and let $m_1 = \lceil 2 t /\delta  \rceil+1$. 
It follows that for some non-random $0\leq k \leq m_1$,
\begin{align*}
\limsup_{N\to \infty} \P\left(\cX^N_{k\delta}(y^\ast) - 
\P[\cX](Y^{k\delta}_{k\delta}=y^\ast)   \geq 2\eps_1/(3n)
\right) \geq p_1(1-p_1)/m_1.
\end{align*}
This contradicts Remark \ref{n13.1} applied at the time $ k\delta$. 
The contradiction completes the proof.
\end{proof}

Recall the Prokhorov distance between probability measures on the Skorokhod space (see \cite[p.~238]{Bill}). 
Convergence in the Prokhorov distance is equivalent to the weak convergence of measures.

\begin{corollary}\label{n25.1}
Assume that $\cX^N_0\Rightarrow \cX$ as $N\to\infty$ for some probability measure
 $\cX$ on $F$.
Fix any $t_1, t_2 >0$.
Consider a Fleming-Viot process $\{\wh \X_t, t\geq s\}$ with $N$ particles, where $s\in [0, t_1]$. 
For any $\eps >0$ there exist $\delta>0$ and $N_1$ such that for all $N\geq N_1$, every $1\leq m \leq N$,
$s\in [0, t_1]$ and  all initial distributions of $\wh \X$ satisfying $\left|\wh \cX _s(y) - \P[\cX](Y^s_s=y)\right| \leq \delta $ for all $y\in F$, the Prokhorov distance between 
  the distribution of $\{\wh M^m_t, t\in[s,s+t_2]\}$ and the distribution 
 of the Poisson process with intensity $\lambda_t$, given by 
 \eqref{n6.1}, on the interval $[s, s+t_2]$ is less than $\eps$.

\end{corollary}

\begin{proof}

Suppose that the corollary is false. Then there exist $\eps>0$ and sequences $(\wh \X^k)_{k\geq 1}$, $(s_k)_{k\geq 1}$, $(\delta_k)_{k\geq 1}$ and $(N_k)_{k\geq 1}$, such that we have $s_k \in [0,t_1]$ for all $k$, $N_k \to \infty$,  $\delta_k \to 0$,
$\left|\wh \cX^{k} _{s_k}(y) - \P[\cX](Y^{s_k}_{s_k}=y)\right| \leq \delta_k $ and the Prokhorov distance between 
  the distribution of $\{\wh M^{k,m}_t, t\in[s_k,s_k+t_2]\}$ and the distribution 
  of the Poisson process with variable intensity $\lambda_t$ 
on the interval $[s_k, s_k+t_2]$ is greater than $\eps$. By compactness, we can find a convergent subsequence of $(s_k)_{k\geq 1}$. By abuse of notation, we will assume that $s_k\to s_\infty \in [0,t_1]$. This and the continuity of the transition probabilities of $Y$ imply that $\left|\wh \cX^{k} _{s_k}(y) - \P[\cX](Y^{s_\infty}_{s_\infty}=y)\right| \to 0 $ for all $y\in F$. 
Let $\X^{N_k}_t := \wh \X_t^{N_k} \circ \theta_{s_k}$.
An application of Proposition \ref{n2.1} to processes $\X^{N_k}$ shows that  the distribution of $\{\wh M^{k,m}_t, t\in[s_k,s_k+t_2]\}$ converges to the distribution 
  of the Poisson process with variable intensity $\lambda_{t+s_\infty}$ 
on the interval $[0, t_2]$. This contradicts the assumption made at the beginning of the proof.
\end{proof}

\begin{theorem}\label{thm:main2}
Assume that $\cX$ is a probability measure on $F$ with $\cX(x) >0$ for all $x\in F$. Suppose that
$\cX^N_0\Rightarrow \cX$ as $N\to\infty$.
  The distribution of $M^J$ converges to the distribution 
  of the Poisson process with intensity $2\lambda_t$ when $N\to 
  \infty$, where $\lambda_t$ is given by \eqref{n6.1}.
\end{theorem}

We have assumed that $\cX(x) >0$ for all $x\in F$ for technical reasons. We expect the theorem to hold without this assumption.

\begin{proof}[Proof of Theorem \ref{thm:main2}]

Fix any $t_1 >0$ and let $m_1  = \lceil t_1 /\delta \rceil+1$, where 
$\delta$ will be specified later.
It would suffice to prove the following assertions.
\begin{enumerate}

\item
For every  $\eps >0$ there exists $\delta_1 >0$ such that
for every $\delta\in(0,\delta_1)$ there exists $N_1$ such that for $N\geq N_1$ and $m=1,2, \dots, m_1$, 
\begin{align}\label{n8.1}
&\P\left(  M^J_{m\delta} -  M^J_{(m-1)\delta} =1 \mid \F_{(m-1)\delta}\right)
\in [(1-\eps)2\delta\lambda_{(m-1) \delta } ,
(1+\eps)2\delta\lambda_{(m-1) \delta } ].
\end{align}

\item
There exist $c$ and $\delta_1 >0$ such that
for all $\delta\in(0,\delta_1)$ there exists $N_1$ such that for $N\geq N_1$ and $m=1,2, \dots, m_1$, 
\begin{align}
\P\left(  M^J_{m\delta} -  M^J_{(m-1)\delta} >1 \mid \F_{(m-1)\delta}\right)
\leq c \delta^2 .\label{n8.2}
\end{align}

\end{enumerate}

Our strategy will be to prove estimates of the type \eqref{n8.1}-\eqref{n8.2} but our argument will be a little bit more complicated.


Since $F$ is finite, we have  
\begin{align}\label{n29.1}
c_1 := n \sup_{x\in F, y\in E} q_{xy} <\infty.
\end{align}

 It is easy to see that for every probability measure $\cX$ on $F$ with $\cX(x) >0$ for all $x\in F$ and
$\eps>0$ there exist $c_2,\delta_1>0$ and $c_3<\infty$ such that for all $\delta \in (0, \delta_1]$, $t\in [0, 2t_1]$, and $x\in F$, if $Y_0$ has the distribution $\cX$ then, 
\begin{align}\label{n14.2}
\P(Y^t_t =x) &\geq c_2,\\
\label{n13.4}
\sup_{s\in[t, t+\delta]}\P(Y^s_s =x) &\leq 
(1 + \eps) \inf_{s\in[t, t+\delta]}\P(Y^s_s =x),\\
\label{n13.3}
\sup_{s\in[t, t+\delta]}\lambda_s &\leq (1 + \eps)\inf_{s\in[t, t+\delta]}\lambda_s ,\\
c_2 &\leq \lambda_t \leq c_3. \label{n18.2}
\end{align}

Let $M_t = \sum_{m=1}^N M^m_t$.
By \eqref{n13.2} and \eqref{n13.3},  for any $\eps>0$, some $\delta_1>0$,
for every  $\delta\in (0,  \delta_1)$ there exists $N_1$ so large that for $N\geq N_1$ and $1 \leq m \leq m_1$, 
\begin{align}\label{n16.1}
\P\left((1-\eps) N \delta\lambda_{(m-1)\delta} \leq M_{m\delta} - M_{(m-1)\delta} \leq (1+\eps) N \delta \lambda_{(m-1)\delta}\right) \geq 1- \delta^4.
\end{align}

By \eqref{n13.2}, for any $\eps>0$, some $\delta_1>0$,
for every  $\delta\in (0,  \delta_1)$ there exists $N_1$ so large that for $N\geq N_1$ and $1 \leq m \leq m_1$, 
\begin{align}\label{n13.5}
\P\left(\sup_{x\in F} \sup_{s\in [(m-1)\delta ,m\delta]} \frac{| \cX^N_s(x) - \P(Y^s_s =x)| }{\P(Y^s_s =x)} \leq \eps \right)
\geq 1-\delta ^4,
\end{align}
and, trivially following from the last formula,
\begin{align}\label{n14.1}
\P\left( \sup_{x\in F} \frac{| \cX^N_{(m-1)\delta}(x) - \P(Y^{(m-1)\delta}_{(m-1)\delta} =x)| }{\P(Y_{(m-1)\delta} ^{(m-1)\delta} =x)} \leq \eps\right)
\geq 1-\delta ^4.
\end{align}
Let $A_1$ be the event in the last formula.

Let $\wt R_1$ be the number of $k$ such that there was exactly one 
branching event and at least one jump along $H^k_{m\delta}$ on the 
interval $[(m-1)\delta ,m\delta]$, and $\chi(k, m\delta, m\delta) 
=\chi(k, m\delta, (m-1)\delta)$. Recall the definition of $c_1$ from 
\eqref{n29.1}.
The intensity of jumps of any process $X^j$ at any position is 
bounded by $c_1< \infty$. It follows that for any $\eps>0$, the 
probability that a process $X^j$ will jump at least once on the 
interval $[(m-1)\delta ,m\delta]$ and some other process will jump 
onto $X^j$ on the same interval (i.e., $U^{i_k}_k = j$ for some 
$\tau_k \in [(m-1)\delta ,m\delta]$) is bounded by $2 c_1 \delta 
\cdot 2c_1  \delta $, for small $\delta$. We can assume that 
$\eps<1$ in \eqref{n16.1} so we see that there exists $c_4$ such that 
for some 
$\delta_1>0$ and
all  $\delta\in (0,  \delta_1)$  there exists $N_1$ so large that for $N\geq N_1$ and $1 \leq m \leq m_1$, 
\begin{align*}
 \E \wt R_1  &\leq c_4 N\delta^2 , \qquad
\E  \wt R_1^2 
 \leq c_4 N\delta^2.
\end{align*}
This and \eqref{n18.2} imply that for any $\eps>0$ , some $\delta_1>0$,
for every  $\delta\in (0,  \delta_1)$,  there exists $N_1$ so large 
that for $N\geq N_1$ and $1 \leq m \leq m_1$, 
\begin{align}\label{n18.1}
\P(\wt R_1 \geq  N  \eps\delta ) 
=\P\left(\wt R_1^2\geq(N\eps\delta)^2\right)
&\leq\frac{c_4 N\delta^2 } {(N \eps\delta)^2}
= \frac{c_4 } {N \eps^2 }
\leq \delta ^4.
\end{align}

Let $\wh R_1$ be the number of $k$ such that there was exactly one 
branching event and at least one jump along $H^k_{m\delta}$ on the 
interval $[(m-1)\delta ,m\delta]$, and at least one jump ocurred later 
than the branching event. An argument very similar to that leading to 
\eqref{n18.1} shows that for any $\eps>0$, some $\delta_1>0$, for 
every  $\delta\in (0,  \delta_1)$ there exists $N_1$ so large that 
for $N\geq N_1$ and $1 \leq m \leq m_1$, 
\begin{align}\label{n20.3}
\P(\wh R_1 \geq  N  \eps\delta ) 
\leq \delta ^4.
\end{align}

Let $R_1$ be the number of $k$ such that there was exactly one 
branching event and at least one jump along $H^k_{m\delta}$ on the 
interval $[(m-1)\delta ,m\delta]$.
 Then  $R_1\leq 2 \wt R_1+ \wh R_1$. We combine \eqref{n18.1} and 
 \eqref{n20.3} to conclude that for any $\eps>0$, some $\delta_1>0$, 
 for every  $\delta\in (0,  \delta_1)$ there exists $N_1$ so large 
 that for $N\geq N_1$ and $1 \leq m \leq m_1$, 
\begin{align}\label{n20.4}
\P( R_1 \geq 3 N  \eps\delta ) 
\leq 2\delta ^4.
\end{align}
Since $\eps>0$ is arbitrarily small in the last formula, \eqref{n14.2}, \eqref{n18.2}, \eqref{n14.1} and \eqref{n20.4} imply that for any $\eps>0$, some $\delta_1>0$, for every  $\delta\in (0,  \delta_1)$ there exists $N_1$ so large that for $N\geq N_1$ and $1 \leq m \leq m_1$, 
\begin{align}\label{n22.1}
\P\left( R_1 \geq  N  \eps\delta \lambda_{(m-1)\delta}
\inf_{x\in F} \cX^N_{(m-1)\delta}(x)\right) 
\leq 3\delta ^4.
\end{align}

Recall the set $\cA$, function $\cL$, and the notions of branching and offspring for elements of $\cA$ introduced in Section \ref{ssec:DHPs}. 
We will consider a branching process $\B$ whose first generation of  
individuals consists of
$\alpha\in \cA$ such that $\alpha = \cL(X^i_{(m-1)\delta})$ for some $i$. The branching process includes all descendants $\beta\in \cA$ of the first generation  individuals provided $\beta= \cL(X^j_s)$  for some $j$ and $(m-1)\delta \leq s \leq m \delta$. 
An individual has a pair of offspring on $[(m-1)\delta,m\delta]$ with 
probability not greater than $2c_1\delta$, for small $\delta$. 
Otherwise it has no offspring. It follows that the expected number of 
individuals in the $j$-th generation is bounded by $N 
(2c_1\delta)^{j-1} $.

Let $R_2$ be the number of $k$ such that  there were exactly two 
branching events along $H^k_{m\delta}$ on the interval $[(m-1)\delta 
,m\delta]$. Let $R_3$ be the number of $k$ such that  there were 
exactly three branching events along $H^k_{m\delta}$ on the interval 
$[(m-1)\delta ,m\delta]$.
Let $R_4$ be the number of 
individuals in the fifth and higher generations of $\B$. Unlike in 
the case of $R_1$ we do not impose any conditions on the number of 
jumps.

Note that the inequality $\delta\leq \eps/ (8c_1^2)$
is equivalent to
$N\eps\delta\geq 
N(2c_1\delta)^2+N\eps\delta/2$ so either one implies
that $\P(R_2 \geq N  \eps\delta ) 
\leq \P(R_2 \geq N(2c_1\delta)^2+ N \eps\delta/2 ) $.

Standard formulas imply that
the variance of the number of individuals in the third generation of 
$\B$ is bounded by $4 N (2c_1\delta)^2$ for small $\delta$. This 
implies that for any $\eps>0$, some $\delta_1\in(0, \eps/ (8c_1^2)) $, for 
every  
$\delta\in (0,  \delta_1)$,  there exists $N_1$ so large that for 
$N\geq N_1$ and $1 \leq m \leq m_1$, 
\begin{align}\label{n16.3}
\P&(R_2 \geq N  \eps\delta ) 
\leq \P(R_2 \geq N(2c_1\delta)^2+ N \eps\delta/2 ) 
\leq
\frac{
 4N (2c_1\delta)^2}
{(N \eps\delta/2 )^2}
=
 \frac{64 c_1^2}
{N \eps^2}
\leq \delta ^4.
\end{align}

The inequality $\delta\leq \sqrt\eps/(4c_1^{3/2})$
is equivalent to
$N\eps\delta\geq 
N(2c_1\delta)^3+N\eps\delta/2$ so either one implies
that $\P(R_3 \geq N  \eps\delta ) 
\leq \P(R_3 \geq N(2c_1\delta)^3+ N \eps\delta/2)  $.

The variance of the number of individuals in the fourth generation of 
$\B$ is bounded by $4 N (2c_1\delta)^3$ for small $\delta$, so
for any $\eps>0$, some $\delta_1\in(0, \sqrt\eps/(4c_1^{3/2}))$,  for every  $\delta\in 
(0,  \delta_1)$, there exists $N_1$ so large that for 
$N\geq N_1$ and $1 
\leq m \leq m_1$,  
\begin{align}\label{n20.5}
\P&(R_3 \geq N  \eps\delta ) 
\leq \P(R_3 \geq N(2c_1\delta)^3+ N \eps\delta/2) 
\leq\frac{
 4N (2c_1\delta)^3}
{(N \eps\delta /2)^2} 
= \frac{128 c_1^3 \delta}
{N \eps^2}
\leq \delta ^4.
\end{align}

For some $\delta_1>0$, for every  $\delta\in (0,  \delta_1)$ there exists $N_1$ so large that for $N\geq N_1$ and $1 \leq m \leq m_1$, 
\begin{align*}
\E(R_4 ) &\leq
\sum _{j\geq 4} N (2c_1\delta)^j
\leq 2
N(2c_1\delta)^4 .
\end{align*}
This and \eqref{n18.2} imply that for any $\eps>0$, some $\delta_1>0$, for every  $\delta\in (0,  \delta_1)$ there exists $N_1$ so large that for $N\geq N_1$ and $1 \leq m \leq m_1$, 
\begin{align}\label{n16.4}
\P&(R_4 \geq N \eps\delta ) \leq
\frac{
 2N (2c_1\delta)^4}
{N \eps\delta }  = 4c_1^4 \delta^3/\eps 
\leq \delta ^{5/2}.
\end{align}

The same justification which enabled us to conclude \eqref{n22.1} from \eqref{n20.4} 
also gives the following estimates, based on \eqref{n16.3}, \eqref{n20.5}
and \eqref{n16.4}. For any $\eps>0$ there exists $\delta_1>0$ such that for
all  $\delta\in (0,  \delta_1)$ 
there exists $N_1$ so large that for $N\geq N_1$ and $1 \leq m \leq m_1$, 
\begin{align}\label{n22.2}
\P&\left(R_2 \geq N  \eps\delta \lambda_{(m-1)\delta}\inf_{x\in F} 
\cX^N_{(m-1)\delta}(x)\right) 
\leq \delta ^4,\\
\P&\left(R_3 \geq N  \eps\delta \lambda_{(m-1)\delta}\inf_{x\in F} 
\cX^N_{(m-1)\delta}(x)\right) 
\leq \delta ^4, \label{n22.3} \\
\P&\left(R_4 \geq N \eps\delta \lambda_{(m-1)\delta}\inf_{x\in F} 
\cX^N_{(m-1)\delta}(x)\right)
\leq  \delta ^{5/2} . \label{n22.4}
\end{align}

Let $R_{5,x}$ be the number of particles that jumped to $x$ on the 
interval $[(m-1)\delta ,m\delta]$.
Since the particles which exit from $F$ jump to the position of a uniformly chosen particle in $F$, \eqref{n14.2}, \eqref{n13.4}, \eqref{n16.1} and \eqref{n13.5} imply that for any $\eps>0$, some $\delta_1>0$, for every  $\delta\in (0,  \delta_1)$ there exists $N_1$ so large that for $N\geq N_1$ and $1 \leq m \leq m_1$, 
\begin{align}\label{n13.6}
\P\left( \sup_{x\in F} \frac{|R_{5,x} - \delta 
\lambda_{(m-1)\delta}N\cX^N_{(m-1)\delta}(x)|}{\delta 
\lambda_{(m-1)\delta}N\cX^N_{(m-1)\delta}(x)}
\leq \eps
\right) \geq 1-\delta ^4.
\end{align}
Let $A_2$ be the event in the last formula.

Let $R_{1,x}$ be the number of $k$ such that there was exactly one 
branching event along $H^k_{m\delta}$ on the interval $[(m-1)\delta 
,m\delta]$, and $H^k_{m\delta}(m\delta)=x$.
Note that 
\begin{align*}
R_{5,x} - R_1-  R_2 - R_3 - R_4 \leq R_{1,x} \leq 2 R_{5,x} + R_1.
\end{align*}
This, \eqref{n22.1}, \eqref{n22.2}, \eqref{n22.3}, \eqref{n22.4} and \eqref{n13.6}, imply that 
for any $\eps>0$, some $\delta_1>0$, for every  $\delta\in (0,  \delta_1)$ there exists $N_1$ so large that for $N\geq N_1$ and $1 \leq m \leq m_1$, 
\begin{align}\label{n22.8}
\P\left( \sup_{x\in F} \frac{|R_{1,x} - 2\delta 
\lambda_{(m-1)\delta}N\cX^N_{(m-1)\delta}(x)|}{\delta 
\lambda_{(m-1)\delta}N\cX^N_{(m-1)\delta}(x)}
\leq 5\eps
\right) \geq 1-2\delta ^{5/2}.
\end{align}
Let $A_3$ be the event in the last formula.

Let $R_{6,x}$ be the number of particles that were located at $x$ at 
time $(m-1)\delta$ and jumped from $x$ to some other state on the 
interval $[(m-1)\delta ,m\delta]$. Recall the definition 
\eqref{n29.1} of $c_1$ and that of the event $A_1$ in \eqref{n14.1}. 
We use \eqref{n14.2} and \eqref{n14.1} to see that for any $\eps>0$, 
some $\delta_1>0$, for every  $\delta\in (0,  \delta_1)$ there exists 
$N_1$ so large that for $N\geq N_1$ and $1 \leq m \leq m_1$, 
\begin{align}\label{n20.1}
\P&\left(\sup_{x\in F} \frac{ R_{6,x} }{ 2 c_1 \delta 
N\cX^N_{(m-1)\delta}(x)}
\geq 1 \right)
\leq \P(A_1^c) + 
\P\left( \sup_{x\in F} \frac{ R_{6,x} }{ 2 c_1 \delta 
N\cX^N_{(m-1)\delta}(x)}
\geq 1
\biggm| A_1\right)\\
&\leq \delta^4 + n\max_{x\in F}\E\left(
\frac{c_1 \delta N\cX^N_{(m-1)\delta}(x)}
{( c_1 \delta N\cX^N_{(m-1)\delta}(x))^2} \biggm| A_1\right)
\leq 2 \delta^4. \nonumber
\end{align}
Let $A_4$ be the event on the left hand side in the last formula.

Suppose that $A_2  \cap A_4$ holds. Then for every $x\in F$,
\begin{align*}
N\cX^N_{m\delta}(x)
 \leq N\cX^N_{(m-1)\delta}(x)+
R_{5,x} \leq
N\cX^N_{(m-1)\delta}(x) +
(1+\eps)\delta \lambda_{(m-1)\delta}N\cX^N_{(m-1)\delta}(x)
\end{align*}
and
\begin{align*}
N\cX^N_{m\delta}(x)
 \geq N\cX^N_{(m-1)\delta}(x)-
R_{6,x} \geq
N\cX^N_{(m-1)\delta}(x) -
2c_1\delta N\cX^N_{(m-1)\delta}(x).
\end{align*}
Hence,
\begin{align*}
1- 2c_1\delta \leq
\frac{N\cX^N_{m\delta}(x)}{N\cX^N_{(m-1)\delta}(x)}
\leq
1+(1+\eps)\delta \lambda_{(m-1)\delta}.
\end{align*}
If in addition $A_3$ holds then
\begin{align}\label{n20.2}
\frac{(2-5\eps)\delta \lambda_{(m-1)\delta}}
{1+(1+\eps)\delta \lambda_{(m-1)\delta}}
\leq 
\frac{R_{1,x}}{N\cX^N_{m\delta}(x)}
\leq 
\frac{(2+5\eps)\delta \lambda_{(m-1)\delta}}
{1- 2c_1\delta}.
\end{align}

Let $A_5 = A_2\cap A_3 \cap A_4$.
By \eqref{n13.6}, \eqref{n22.8} and \eqref{n20.1}, we have 
$\P(A_5^c) \leq 3 \delta^{5/2}$ for small $\delta$,
so for every $m$ there exists an event $B_m \in \F_{(m-1)\delta}$ such that
\begin{align}\label{n26.1}
\P(B_m) \geq 1 - 2 \delta^{5/4}
\end{align}
and on the event $B_m$,
\begin{align}\label{n22.7}
 \E\left( \bl_{ A_5^c}  \mid \F_{(m-1)\delta}\right) \leq 2 \delta^{5/4}.
\end{align}

Let $C_k^x$ be the intersection of  $\{H^k_{m\delta}(m\delta) = x\}$ and the event that there was exactly one branching event along $H^k_{m\delta}$ on the interval $[(m-1)\delta ,m\delta]$. Let $C_J$ be the  event that there was exactly one branching event along the spine on the interval $[(m-1)\delta ,m\delta]$.
Let $D^x_k = \{\cL(J^N,m\delta) = \cL(X^k_{m\delta}), J^N_{m\delta}=x\}$.
In the following calculation we use
the Markov property applied at time $m\delta$ and \eqref{n20.2},
\begin{align*}
\P(&C_J \mid \F_{(m-1)\delta}) 
= \sum_{x\in F}\sum_{k=1}^N \P(D^x_k \cap C_k^x \mid \F_{(m-1)\delta})\\
&= \E\left(\sum_{x\in F}\sum_{k=1}^N \P(D^x_k \cap C_k^x
\mid \F_{m\delta}) \biggm| \F_{(m-1)\delta}\right)\\
&= \E\left(\sum_{x\in F}\sum_{k=1}^N \bl_{C_k^x} \P(D^x_k 
\mid \F_{m\delta}) \biggm| \F_{(m-1)\delta}\right)\\
&= \E\left(\sum_{x\in F}\sum_{k=1}^N \bl_{C_k^x\cap A_5} \P(D^x_k 
\mid \F_{m\delta}) \biggm| \F_{(m-1)\delta}\right)\\
&\quad + \E\left(\sum_{x\in F}\sum_{k=1}^N \bl_{C_k^x\cap A_5^c} \P(D^x_k 
\mid \F_{m\delta}) \biggm| \F_{(m-1)\delta}\right)\\
&\leq  \E\left(\sum_{x\in F}\sum_{k=1}^N \bl_{C_k^x\cap A_5} \P(D^x_k 
\mid \F_{m\delta}) \biggm| \F_{(m-1)\delta}\right)
+ \E\left( \bl_{ A_5^c}  \mid \F_{(m-1)\delta}\right)\\
&\leq \E\left(\sum_{x\in F}\sum_{k=1}^N \bl_{C_k^x\cap A_5} 
\frac{\P(J^N_{m\delta} = x \mid \F_{m\delta})}{ N\cX^N_{m\delta}(x)} 
\biggm| \F_{(m-1)\delta}\right)
+ \E\left( \bl_{ A_5^c}  \mid \F_{(m-1)\delta}\right)\\
&= \E\left(\sum_{x\in F} \frac{\P(J^N_{m\delta} = x \mid 
\F_{m\delta})}{N\cX^N_{m\delta}(x)} \sum_{k=1}^N \bl_{C_k^x\cap A_5} 
\biggm| \F_{(m-1)\delta}\right)
+ \E\left( \bl_{ A_5^c}  \mid \F_{(m-1)\delta}\right)\\
&= \E\left(\sum_{x\in F} \frac{\P(J^N_{m\delta} = x \mid 
\F_{m\delta})}{N\cX^N_{m\delta}(x)} R_{1,x} \bl_{A_5} \biggm| 
\F_{(m-1)\delta}\right)
+ \E\left( \bl_{ A_5^c}  \mid \F_{(m-1)\delta}\right)\\
&\leq \E\left(\sum_{x\in F} \P(J^N_{m\delta} = x \mid \F_{m\delta})
\frac{(2+5\eps)\delta \lambda_{(m-1)\delta}}
{1- 2c_1\delta}
 \biggm| \F_{(m-1)\delta}\right)
+ \E\left( \bl_{ A_5^c}  \mid \F_{(m-1)\delta}\right)\\
&= \frac{(2+5\eps)\delta \lambda_{(m-1)\delta}}
{1- 2c_1\delta} \E\left(\sum_{x\in F} \P(J^N_{m\delta} = x \mid \F_{m\delta})
 \biggm| \F_{(m-1)\delta}\right)
+ \E\left( \bl_{ A_5^c}  \mid \F_{(m-1)\delta}\right)\\
&=\frac{(2+5\eps)\delta \lambda_{(m-1)\delta}}
{1- 2c_1\delta}
+ \E\left( \bl_{ A_5^c}  \mid \F_{(m-1)\delta}\right) .
\end{align*}
This, \eqref{n18.2} and \eqref{n22.7} imply that
for any $\eps>0$, some $\delta_1>0$, for every  $\delta\in (0,  \delta_1)$ there exists $N_1$ so large that for $N\geq N_1$ and $1 \leq m \leq m_1$, 
 on the event $B_m$, 
\begin{align}\nonumber
\P(C_J \mid \F_{(m-1)\delta}) 
&\leq
\frac{(2+5\eps)\delta \lambda_{(m-1)\delta}}
{1- 2c_1\delta}
+ \E\left( \bl_{ A_5^c}  \mid \F_{(m-1)\delta}\right)
\leq  \frac{(2+5\eps)\delta \lambda_{(m-1)\delta}}
{1- 2c_1\delta} + 2 \delta^{5/4}\\
&\leq (2+6\eps)\delta \lambda_{(m-1)\delta}.\label{n22.10}
\end{align}

Let $C'_J$ be the  event that there were at least two branching events along the spine on the interval $[(m-1)\delta ,m\delta]$. A calculation similar to that for $C_J$ but using \eqref{n22.2}-\eqref{n22.4} instead of \eqref{n22.8} shows that there exist events  $B'_m \in \F_{(m-1)\delta}$ such that
$\P(B'_m) \geq 1 - 2 \delta^{5/4}$ and for any $\eps>0$, some $\delta_1>0$, for every  $\delta\in (0,  \delta_1)$ there exists $N_1$ so large that for $N\geq N_1$ and $1 \leq m \leq m_1$, 
 on the event $B'_m$, 
\begin{align}\label{n22.11}
\P(C'_J \mid \F_{(m-1)\delta}) 
&\leq 3\eps\delta \lambda_{(m-1)\delta}.
\end{align}

We now derive the lower estimate for $\P(C_J \mid \F_{(m-1)\delta}) $,
\begin{align*}
\P(&C_J \mid \F_{(m-1)\delta}) 
= \sum_{x\in F}\sum_{k=1}^N \P(D^x_k \cap C_k^x \mid \F_{(m-1)\delta})\\
&= \E\left(\sum_{x\in F}\sum_{k=1}^N \bl_{C_k^x\cap A_5} \P(D^x_k 
\mid \F_{m\delta}) \biggm| \F_{(m-1)\delta}\right)\\
&\qquad + \E\left(\sum_{x\in F}\sum_{k=1}^N \bl_{C_k^x\cap A_5^c} \P(D^x_k 
\mid \F_{m\delta}) \biggm| \F_{(m-1)\delta}\right)\\
&\geq  \E\left(\sum_{x\in F}\sum_{k=1}^N \bl_{C_k^x\cap A_5} \P(D^x_k 
\mid \F_{m\delta}) \biggm| \F_{(m-1)\delta}\right)\\
&= \E\left(\sum_{x\in F}\sum_{k=1}^N \bl_{C_k^x\cap A_5} 
\frac{\P(J^N_{m\delta} = x \mid \F_{m\delta})}{ N\cX^N_{m\delta}(x)} 
\biggm| \F_{(m-1)\delta}\right)\\
&= \E\left(\sum_{x\in F} \frac{\P(J^N_{m\delta} = x \mid 
\F_{m\delta})}{N\cX^N_{m\delta}(x)} \sum_{k=1}^N \bl_{C_k^x\cap A_5} 
\biggm| \F_{(m-1)\delta}\right)\\
&= \E\left(\sum_{x\in F} \frac{\P(J^N_{m\delta} = x \mid 
\F_{m\delta})}{N\cX^N_{m\delta}(x)} R_{1,x} \bl_{A_5} \biggm| 
\F_{(m-1)\delta}\right)\\
&\geq \E\left(\sum_{x\in F} \P(J^N_{m\delta} = x \mid \F_{m\delta})
\frac{(2-5\eps)\delta \lambda_{(m-1)\delta}}
{1+(1+\eps)\delta \lambda_{(m-1)\delta}}  \bl_{A_5}
 \biggm| \F_{(m-1)\delta}\right)\\
&= \frac{(2-5\eps)\delta \lambda_{(m-1)\delta}}
{1+(1+\eps)\delta \lambda_{(m-1)\delta}} \E\left( \bl_{A_5}\sum_{x\in F} \P(J^N_{m\delta} = x \mid \F_{m\delta})
 \biggm| \F_{(m-1)\delta}\right)\\
&=\frac{(2-5\eps)\delta \lambda_{(m-1)\delta}}
{1+(1+\eps)\delta \lambda_{(m-1)\delta}} \E\left( \bl_{A_5}
 \mid \F_{(m-1)\delta}\right) .
\end{align*}
This, \eqref{n18.2} and \eqref{n22.7} imply that
for any $\eps>0$, some $\delta_1>0$, for every  $\delta\in (0,  \delta_1)$ there exists $N_1$ so large that for $N\geq N_1$ and $1 \leq m \leq m_1$, 
 on the event $B_m$, 
\begin{align}\label{n22.12}
\P(C_J \mid \F_{(m-1)\delta}) 
&\geq
\frac{(2-5\eps)\delta \lambda_{(m-1)\delta}}
{1+(1+\eps)\delta \lambda_{(m-1)\delta}} \E\left( \bl_{A_5}
 \mid \F_{(m-1)\delta}\right)\\
&\geq  \frac{(2-5\eps)\delta \lambda_{(m-1)\delta}}
{1+(1+\eps)\delta \lambda_{(m-1)\delta}}
(1- 2 \delta^{5/4})
\geq (2-6\eps)\delta \lambda_{(m-1)\delta}.\nonumber
\end{align}

Let $B_* = \bigcap_{1 \leq m \leq m_1} (B_m \cap B'_m)$. We have 
\begin{align*}
\P(B_*^c) \leq m_1 4 \delta^{5/4} =
(\lceil t_1 /\delta \rceil+1) 4 \delta^{5/4}
\leq c_5 \delta^{1/4}.
\end{align*}

Let $\wt M$ be a Poisson process with intensity $2\lambda_t$,
independent of $M^J$. We define $\wh M^J $ 
by setting $\wh M^J_t - \wh M^J_{(m-1)\delta} = M^J_t -  M^J_{(m-1)\delta}$
for $t\in [(m-1)\delta, m\delta]$ on the event $B'_m \cap B_m$ for $m\geq 1$.
We let $\wh M^J_t - \wh M^J_{(m-1)\delta} = \wt M_t - \wt M_{(m-1)\delta}$
for $t\in [(m-1)\delta, m\delta]$ on $(B'_m \cap B_m)^c$ for $m\geq 1$.
It follows from \eqref{n22.10}, \eqref{n22.11} and \eqref{n22.12} that
$\wh M^J$ satisfies \eqref{n8.1}-\eqref{n8.2}. Hence, the processes $\wh M^J$ converge to the Poisson process with  intensity $2\lambda_t$ as $N\to \infty$.
Since $\P(B_*^c)$ can be made arbitrarily small by choosing small $\delta$, we conclude that  the processes $ M^J$ converge to the Poisson process with  intensity $2\lambda_t$ as $N\to \infty$.
\end{proof}

The next theorem is concerned with the distribution of a side 
branch of the spine of Fleming--Viot process. To this aim we consider 
two branching processes. The first process, $\V$, will be the 
branching version of $Y$ with the deterministic branching rate 
$\lambda_t$ defined in \eqref{n6.1}. Then we will define a branching 
process $\Z$ representing descendants (along historical paths) of 
one  of the components of $\X$. The constructions are routine but 
tedious so we will only sketch them.

Fix any  probability measure $\cX$ on $F$.
 Given $x_1\in F$ and $t_1 \geq 0$, let $\{\wh Y_t, t \geq t_1\}$ have the distribution of the process $Y$ started from $x_1$ at time $t_1$. 
Let $\tau_F$ be the exit time of $\wh Y$ from $F$.
Let $U$ be an independent random variable with the distribution given by 
$\P(U> u) = \exp\left(-\int_{t_1} ^ u \lambda_t dt\right)$ for $u \geq t_1$. 
We let $\zeta = \tau_F \land U$. If $\tau_F < U$ then we let $\bfa = 0$. Otherwise $\bfa = 1$. Let $\cP(t_1, x_1)$ denote the distribution of $\left(\{\wh Y_t, t_1 \leq t < \zeta\}, \bfa\right)$.

Let $\cB$ be the family of sequences of the form $(i_1, i_2, \dots, i_k)$, where $i_1 = 0$ and each $i_j$ is either 0 or 1. 
If $\beta = (i_1, i_2, \dots, i_k)$ then we will write $\beta+ 0 = (i_1, i_2, \dots, i_k,0)$ and $\beta+ 1 = (i_1, i_2, \dots, i_k,1)$. We will say that $\beta+0$ and $\beta +1$ are offspring of $\beta$.

Fix any $x_1 \in F$ and $t_1 \geq 0$.
There exists a branching process $\V=\V^{t_1, x_1}$ starting from a single individual with the following properties.
Individuals $V^\beta$ in $\V$ are indexed by  $\beta \in \cB$. Every individual $V^\beta$ is a process $\{V^\beta_t, s_\beta \leq t < t_\beta\}$ for some $0\leq s_\beta < t_\beta < \infty$.
Let $\cB_\V$ denote the random set of all indices of all individuals in $\V$.
We always have $(0) \in \cB_\V$. 
We call $V^\gamma$ an offspring of $V^\beta$ if and only if $\gamma$ 
is an offspring of $ \beta $.
If $\beta \in \cB_\V$ then all ancestors of $\beta$ are also in $\cB_\V$.  
If $\beta \in \cB_\V$ has no offspring in $\cB_\V$ then we let $\bfa_\beta = 0$. Otherwise $\bfa_\beta = 1$. 
If $\gamma $ is an offspring of $\beta$ then $V^\gamma_{s_\gamma} = 
V^\beta_{t_\beta -}$. If $\beta = (0)$ then
the distribution of $\left(\{V^\beta_t, s_\beta \leq t < t_\beta\}, \bfa_\beta\right)$  is $\cP(t_1, x_1)$. For any other $\beta\in \cB_\V$,
the conditional distribution of $\left(\{V^\beta_t, s_\beta \leq t < t_\beta\}, \bfa_\beta\right)$ given the distribution of all ancestors of $V^\beta$ is that of $\cP(s_\beta, V^\beta_{s_\beta})$. Let $\calD(t_1, x_1)$ denote the distribution of $\V^{t_1, x_1}$.

\begin{remark}\label{n23.1}
We will now argue that the process $\V$ has a finite lifetime a.s.
Let $K(t)$ be the number of individuals at time $t$.
Suppose that the process $\V$ starts at time 0 and its starting distribution is
randomized so that the position of the unique individual 
at time 0 has distribution $\cX$. The branching intensity $\lambda_t$ has been
chosen so that the expected number of individuals is constant in time for this initial distribution, i.e.,
$\E K(t) = 1$ for all $t\geq 0$. This implies that $K(t)$ cannot grow to infinity (in finite or infinite time) with positive probability. It follows that, with probability 1, for some $c_1< \infty$, there will be arbitrarily large times $t_k$ with $K(t_k) \leq c_1$. A standard argument based on the strong Markov property shows that $\V$ has to become extinct within one unit of time of one of $t_k$'s (or earlier), a.s. Since this is an almost sure result, it is easy to see that it implies that $\V^{s, x}$ has a finite lifetime, a.s., for every $s\geq 0$ and $x\in F$.
\end{remark}

Suppose that $X^k_{t_1} = x_1$ and let $\alpha_0 = \chi(k, t_1, t_1)\in \cA$.
Let $\cA_\Z$ be the family of all descendants $\alpha $ of $\alpha_0$ in $ \cA$
such that $\alpha = \cL(X^j_t)$ for some $j\in\{1,\dotsc,N\}$ and 
$t\geq t_1$. It is elementary to see that there exists a one to one 
mapping  $\Gamma: 
\cA_\Z \to \cB$ with $\Gamma (\alpha_0) = (0)$, preserving 
parenthood, i.e., 
$\Gamma(\alpha_1) $ is a parent of $\Gamma(\alpha_2)$ if and only if 
$\alpha_1$ is a parent of $\alpha_2$. We choose such a mapping 
$\Gamma$ in an arbitrary way.
Let $\cB_\Z = \Gamma(\cA_\Z)$. We let $\Z^{k,t_1}(t)$ be a branching 
process with individuals $Z^\beta$ for $\beta \in \cB_\Z$.
We call $Z^\beta$ an offspring of $Z^\gamma$ if and only if $\beta$ is an offspring of $ \gamma $. 
Every individual $Z^\beta$ is a process $\{Z^\beta_t, s_\beta \leq t < t_\beta\}$ for some $0\leq s_\beta < t_\beta < \infty$. 
If $\beta = \Gamma(\alpha)$ and $\alpha = ((a_1,b_1), (a_2,b_2), \dots, (a_m,b_m))$ then $s_\beta = \inf\{t\geq 0:   \chi(a_m, t,t) = \alpha\}$,
$t_\beta = \sup\{t\geq 0:  \chi(a_m, t,t) = \alpha\}$
and $Z^\beta_t = X^{a_m}_t$ for $t\in [s_\beta, t_\beta)$.
 Note that $\Z^{k,t_1}$ has an infinite lifetime if and only if the spine $J^N$ is a part of this branching process.  

Consider any $k\geq 1$ and let $u_k$ be the time of the $k$-th branching point of the spine $J^N$. Suppose that $\chi(J, u_k) = j_1$ and note that
there is a unique $j_2 \in\{1,\dots, N\}$ such that $j_2\ne j_1$ and 
$\chi(j_2, u_k, t) = \chi(j_1, u_k, t)$ for all $t< u_k$. Let 
$\Z_k=\Z^{j_2,u_k}$.

\begin{theorem}\label{thm:main3}
Assume that $\cX$ is a probability measure on $F$ with $\cX(x) >0$ for all $x\in F$. Suppose that
$\cX^N_0\Rightarrow \cX$ as $N\to\infty$ and consider the process $Y^\infty$ with  the initial distribution  $\cX$.
Fix any $k\geq 1$ and let $\mu_k(dt)$ be the distribution of  the time 
of the $k$-th jump of a Poisson process with intensity $\lambda_t$, independent of $Y^\infty$. 
 Then the distribution of $\Z_k$ converges, as $N\to \infty$, to 
 $\int_0^\infty 
\sum_{x\in F}  \calD(t,x) \P(Y^\infty_t =x) \mu_k(dt)$.
\end{theorem}

\begin{proof}

We will use notation and definitions from the proof of Theorem \ref{thm:main2}.
Let
\begin{align*}
A'_1=
\left\{ \sup_{x\in F} \frac{| \cX^N_{m\delta}(x) - \P(Y^{m\delta}_{m\delta} =x)| }{\P(Y^{m\delta}_{m\delta} =x)} \leq \eps\right\},
\end{align*}
and note that this is almost the same as the event $A_1$ in 
\eqref{n14.1} except that $m-1$ is replaced with $m$. 
Recall that $m_1$ is defined to be a function of $\delta$ in the proof of Theorem \ref{thm:main2}, and $\delta$ is specified later in that proof.

Recall that $C_J$ is the  event that there was exactly one branching event along the spine on the interval $[(m-1)\delta ,m\delta]$.
We obtain from \eqref{n26.1}, \eqref{n22.10} and \eqref{n22.12}  that
for any $\eps>0$, some $\delta_1>0$, for every  $\delta\in (0,  \delta_1)$ there exists $N_1$ so large that for $N\geq N_1$ and $1 \leq m \leq m_1$, 
\begin{align}\label{n26.2}
\P(C_J ) 
&\geq (2-6\eps)\delta \lambda_{(m-1)\delta} (1 - 2 \delta^{5/4}),\\
\P(C_J ) 
&\leq (2+6\eps)\delta \lambda_{(m-1)\delta} + 2 \delta^{5/4}.
\label{n26.3}
\end{align}
This, \eqref{n18.2} and \eqref{n14.1} imply that 
for any $\eps>0$, some $\delta_1>0$, for every  $\delta\in (0,  \delta_1)$ there exists $N_1$ so large that for $N\geq N_1$ and $1 \leq m \leq m_1$, 
\begin{align}\label{n27.1}
\P(A'_1 \mid C_J) 
& = \frac{\P(A'_1 \cap C_J)}{\P(C_J)}
\geq \frac{\P( C_J) - \P((A'_1)^c )}{\P( C_J) } \\
&\geq
\frac{(2-6\eps)\delta \lambda_{(m-1)\delta} (1 - 2 \delta^{5/4}) - \delta^4}
{(2+6\eps)\delta \lambda_{(m-1)\delta} + 2 \delta^{5/4} }
\geq 1- 7\eps.\nonumber
\end{align}

Fix some $s_1\geq 0$ and $1\leq \ell \leq N$. 
Consider the process $\Z^{\ell,s_1}$ conditioned on $\{X^\ell_{s_1}= Z^{(0)}_{s_1} =x_1\}$. It follows from Proposition \ref{n2.1} that the distribution of the first individual in this process, i.e., $\{Z^{(0)}_t, t\geq s_1\}$ converges, as $N\to \infty$, to the distribution of the process $Y$ killed at the first jump of an independent Poisson process with intensity $\lambda_t$, for $t\geq s_1$. In other words, it converges to the distribution of the first individual in the the process $\V^{s_1, x_1}$. 

Next consider any family of sequences $\{\beta_1 =(0), \beta_2,  \dots, \beta_m\}\subset \cB$ such that if $\beta_j$ belongs to the family then the parent belongs to the family as well. By the Markov property applied at the death times of individuals, for any such family, for every $\beta_j$, Proposition \ref{n2.1} implies that the distribution of the individual labeled $\beta_j$ in $\Z^{\ell,s_1}$ converges to the distribution of the individual with the same label in 
$\V^{s_1, x_1}$. Moreover, we have  convergence of the joint distribution of all individuals labeled $\beta_1, \beta_2,  \dots, \beta_m$ in the process $\Z^{\ell,s_1}$  to the joint distribution of the similarly labeled individuals in 
$\V^{s_1, x_1}$. By Remark \ref{n23.1}, the process $\V^{s_1, x_1}$ has a finite lifetime so $\cB_\V$ is finite, a.s. This completes the proof  that the distribution of $\Z^{\ell,s_1}$ converges to that of $\V^{s_1, x_1}$.

By Theorem \ref{thm:main1} and \eqref{n14.1},
for any $\eps>0$, some $\delta_1>0$, for every 
$\delta\in(0,\delta_1)$, there exists $N_1$ so large that for all  
$N\geq N_1$, $1\leq \ell \leq N$ and $1 \leq m \leq m_1$, 
\begin{align}\label{n27.2}
\P(\chi(J,m\delta) = \ell) < \eps.
\end{align}

We obtain from the above argument concerning the distribution of $\Z^{\ell,s_1}$  combined with the Markov property applied at $m\delta$, Corollary \ref{n25.1}, \eqref{n27.1} and \eqref{n27.2}  that
for any $\eps>0$, some $\delta_1>0$, for every  $\delta\in (0,  \delta_1)$ there exists $N_1$ so large that for all $x\in F$, $N\geq N_1$, $1\leq \ell \leq N$ and $1 \leq m \leq m_1$, conditional
on $C_J$  and  the event $\{Z^\ell_{m\delta} =x, \chi(J,m\delta) \ne \ell\}$, the Prokhorov distance between the distribution of $\Z^{\ell,m\delta}$ and that of $\V^{m\delta, x}$ is less than $\eps$.

Since the definition of $j_2$ (an element of the definition of 
$\Z_k$) does not refer to the post-$u_k$ process, the claim made in 
the last paragraph applies not only to a fixed $\ell$ but also to 
$j_2$. Hence, 
for any $\eps>0$, some $\delta_1>0$, for every  $\delta\in (0,  \delta_1)$, there exists $N_1$ so large that for all $N\geq N_1$ and $1 \leq m \leq m_1$, conditional
on $C_J$, the Prokhorov distance between the distribution of $\Z_k$ 
and that of $\V^{m\delta, x}$ is less than $\eps$.
This easily implies the theorem.
\end{proof}

\begin{remark}
It is not hard to see that the following ``propagation of chaos'' assertion holds: 
for any fixed $k$, the  processes $\Z_{1}, 
\Z_{2}, \dots, \Z_k$ are asymptotically independent, when 
$N\to\infty$.
\end{remark}

\section{Spine distribution for a fixed $N$}\label{fixedN}

The main result of this paper states that if the number $N$ of
particles of a Fleming-Viot process increases to infinity then
the distribution of the spine converges to the distribution of the
underlying Markov process conditioned not to hit the boundary. 
One could wonder whether 
the theorem must have the asymptotic character; perhaps the claim
is true for every fixed $N$.
This section is devoted to an
example of a Fleming-Viot process with $N=2$ particles such that the distribution of the spine is not
the same as the distribution of the driving 
process conditioned on non-extinction.

Let $Y_t$ be a continuous--time Markov process $Y_t$ with
the state space $E=\set{0,1,2}$, $F=\set{1,2}$ and the transition rate matrix 
\begin{equation*}
  A=
  \begin{bmatrix}
    0 & 0 & 0 \\
    4 & -6 & 2 \\
    1 & 6 & -7
  \end{bmatrix}.
\end{equation*}
Let $\X_t=(X_t^1,X_t^2)$ denote the 2-particle Fleming-Viot process based on $Y$. Then $\X$ has the state space 
$\set{(1,1),(1,2),(2,1),(2,2)}$ and the transition rate matrix
\begin{equation*}
  \A=
  \begin{bmatrix}
    -4 & 2 & 2 & 0 \\
    7 & -13 & 0 & 6 \\
    7 & 0 & -13 & 6 \\
    0 & 6 & 6 & -12
  \end{bmatrix}.
\end{equation*}
Therefore the stationary distribution $\pi$ of $\X$ determined by
$\pi \A=0$ is 
\begin{equation}\label{o27.1}
  \pi=\left( \frac{7}{13},  \frac{2}{13}, \frac{2}{13}, \frac{2}{13}
  \right).
\end{equation}

Let  $Y'$ and $Y''$ be independent copies of
$Y$. The state space of $(Y', Y'')$ is
\begin{equation*}
  \Lambda=\set{(0,0),(0,1),(0,2), (1,0),(2,0),(1,1),(1,2),(2,1), (2,2)}
\end{equation*}
and the transition rate matrix for $(Y', Y'')$ is
\begin{equation*}
  \B=
  \begin{bmatrix}
    0&&&&\cdots&&&&0\\
    \vdots&&&&&&&&\vdots\\
    0&&&&\cdots&&&&0\\
    0&4&0&4&0&-12&2&2&0\\
    0&0&4&1&0&6&-13&0&2\\
    0&1&0&0&4&6&0&-13&2&\\
    0&0&1&0&1&0&6&6&-14
  \end{bmatrix}.
\end{equation*}
Let
\begin{equation*}
  f(x,y)=\P\left( \text{$Y''$ reaches 0 before $Y'$} \mid
  Y'_0=x,Y''_0=y \right),\quad x,y=1,2.
\end{equation*}
Then $f$ is harmonic with respect to $\B$, i.e., for all $(x_1,x_2)\in\Lambda$,
\begin{equation*}
  \sum_{(y_1,y_2)\in\Lambda} \B((x_1,x_2),(y_1,y_2))\left( f(y_1,y_2)-f(x_1,x_2) \right) =0.
\end{equation*}
It follows from the definition of $f$ that $f(0,1)=f(0,2)=0$ and $f(1,0)=f(2,0)=1$.
By symmetry, $f(1,1)=\frac{1}{2}$ and $ f(2,2)=\frac{1}{2}$.
It is elementary to check that
\begin{equation*}
   f(1,2)=\frac{5}{13},\qquad
  f(2,1)=\frac{8}{13}.
\end{equation*}

Let $J_t$ denote the spine of $\X_t$. The spine passes through $X^1_t$ (i.e., $\chi(J, t) =1$)
if an only if $X^2$ ``jumps to 0'' before $X^1$, after time $t$. The probability of this event is the same as the probability that $Y''$ will hit 0 before $Y'$, assuming that $(Y'_0, Y''_0) = (X^1_t, X^2_t)$.
If $\X_t=(1,1)$ then $J_t=1$, and
if $\X_t=(2,2)$ then $J_t=2$. 
If $\X_t=(1,2)$ then $J_t=1$ with probability
$f(1,2)=\frac{5}{13}$, and, by symmetry,
if $\X_t=(2,1)$ then $J_t=1$ with probability
$\frac{5}{13}$. 
Assume that $\X$ is in the stationary regime and
recall the stationary probabilities for $\X$ given in \eqref{o27.1} to see that
\begin{equation}\label{n30.1}
  \P(J_t=1)=\frac{7}{13}\cdot 1+\frac{2}{13}\cdot\frac{5}{13}
  +\frac{2}{13}\cdot\frac{5}{13}+\frac{2}{13}\cdot
  0=\frac{111}{169}.
\end{equation}

We will show that for a (generic) fixed $t>0$, the distribution of $J_t$ is not the same as the distribution of $Y$ conditioned to stay in $F$ until time $t$, and it is not the same as the distribution of $Y$ conditioned to stay in $F$ forever.

Let $\P_\mu$ denote the distribution of $Y$ with the initial
    distribution $\mu$ and assume that $\mu(0)=0$. 
    The transition probabilities of $Y$ are given by
    \begin{equation*}
      P_t^Y :=\e^{tA}=
     \begin{bmatrix}
      1 & 0 & 0 \\
      1-\frac{6}{7}\e^{-3t}-\frac{1}{7}\e^{-10t} & \frac{4}{7}
      \e^{-3t}+\frac{3}{7}\e^{-10t}  &
      \frac{2}{7} \e^{-3t}-\frac{2}{7}\e^{-10t}  \\
      1-\frac{9}{7}\e^{-3t}+\frac{2}{7}\e^{-10t} & \frac{6}{7}
      \e^{-3t}-\frac{6}{7}\e^{-10t}  &
      \frac{3}{7} \e^{-3t}+\frac{4}{7}\e^{-10t}  \\
     \end{bmatrix},
    \end{equation*}
    so
    \begin{align}\label{o27.2}
      \P_{\mu}(Y_t\neq 0)&=
      \left(\frac{6}{7}\mu(1)+\frac{9}{7}\mu(2)\right)\e^{-3t}+
         \left(\frac{1}{7}\mu(1)-\frac{2}{7}\mu(2)\right)\e^{-10t},\\
      \P_{\mu}(Y_t=1)&=
      \left(\frac{4}{7}\mu(1)+\frac{6}{7}\mu(2)\right)\e^{-3t}+
         \left(\frac{3}{7}\mu(1)-\frac{6}{7}\mu(2)\right)\e^{-10t},
\label{o27.3}\\
      \P_{\mu}(Y_t=2)&=
      \left(\frac{2}{7}\mu(1)+\frac{3}{7}\mu(2)\right)\e^{-3t}+
         \left(\frac{4}{7}\mu(2)-\frac{2}{7}\mu(1)\right)\e^{-10t}. \label{o27.4}
    \end{align}
 It follows that
\begin{align*}
\P_\mu\left( Y_t=1 \mid Y_t\neq 0
      \right)=
      \frac{\P_\mu(Y_t=1)}{\P_\mu(Y_t\neq 0)}
=\frac{2\left( 2\mu(1)+3\mu(2) \right)
	+\left( 3\mu(1)-6\mu(2) \right)\e^{-7t}}
	{3\left( 2\mu(1)+3\mu(2) \right)
	  +\left( \mu(1)-2\mu(2)\right)\e^{-7t}}\to\frac{2}{3}
\end{align*}
    as $t\to\infty$, regardless of the initial distribution $\mu$.
Comparing this value to \eqref{n30.1}, we see that for large $t$ the distribution of $J_t$ (in the stationary regime)
    is not the same
    as the law of $Y_t$ conditioned to stay in $F$ until $t$.

   Next we will compare the distribution of $J_t$ with the
    distribution of the process $Y$ conditioned to stay in $F$ 
forever, i.e., $Y_t^{\infty}$. Since $0$ is an absorbing state,
    \begin{equation*}
      \P_\mu( Y_t^{\infty}=x
      )=\lim_{s\to\infty}\P_\mu(Y_t=x\mid Y_{t+s}\neq 0),\quad
      x=1,2, \ t >0.
    \end{equation*}
By the Markov property,
    \begin{equation}
      \begin{split}
	\P_\mu\left( Y_t=1\mid Y_{t+s}\neq 0 \right)&=
	\frac{\P_\mu(Y_t=1)(1-P_s^{Y}(1,0))}
	{\P_\mu(Y_t=1)(1-P_s^{Y}(1,0))+\P_\mu(Y_t=2)(1-P_s^{Y}(2,0))}\\
	&=\frac{\P_\mu(Y_t=1)}
	{\P_\mu(Y_t=1)+\P_\mu(Y_t=2)
 {\displaystyle \frac{1-P_s^{Y}(2,0)}{1-P_s^{Y}(1,0)}}}.
      \end{split} 
    \end{equation}
 By \eqref{o27.2},
    \begin{equation*}
      \lim_{s\to\infty}\frac{1-P_s^{Y}(2,0)}{1-P_s^{Y}(1,0)}=
      \lim_{s\to\infty}\frac{9-2\e^{-7s}}{6+\e^{-7s}}=\frac{3}{2},
    \end{equation*}
    so this and \eqref{o27.3}-\eqref{o27.4} imply that
    \begin{equation*}
      \P_\mu( Y_t^{\infty}=1
      )=\frac{4}{7}+\frac{6}{7}\cdot\frac{\mu(1)-2\mu(2)}{2\mu(1)+3\mu(2)}\e^{-7t}.
    \end{equation*}
This probability converges to $4/7$ when $t\to \infty$. 
This value is different from that in \eqref{n30.1} so 
for large $t$ the distribution of $J_t$ (in the stationary regime)
    is not the same
    as the law of $Y_t^\infty$.

  \section{Acknowledgments}

We are grateful to Theodore Cox, Simon Harris, Doug Rizzolo
and Anton Wakolbinger for the most helpful advice.

\nocite{*}
\bibliographystyle{abbrv}
\bibliography{FVspine}

\end{document}